\documentclass[11pt]{amsart}
\usepackage{amsfonts}
\usepackage[a4paper,margin=1in]{geometry}
\usepackage[T1]{fontenc}
\usepackage{color}
\usepackage{lmodern}
\usepackage{microtype}
\usepackage{amsmath,amssymb,amsthm,mathtools,mathrsfs}
\usepackage{booktabs}
\usepackage{enumitem}
\usepackage[hidelinks]{hyperref}
\usepackage{comment}

\newtheorem{theorem}{Theorem}[section]
\newtheorem{lemma}[theorem]{Lemma}
\newtheorem{proposition}[theorem]{Proposition}

\theoremstyle{remark}
\newtheorem{remark}[theorem]{Remark}

\begin{document}
\title{\textbf{Optimal concentration in the Paley--Wiener space}}
\author[L. D. Abreu]{Lu\'{\i}s Daniel Abreu}
\address{Lu\'is Daniel Abreu, Faculty of Mathematics \\
University of Vienna \\
Oskar-Morgenstern-Platz 1 \\
1090 Vienna, Austria}
\email{abreuluisdaniel@gmail.com}
\author[M. Speckbacher]{Michael Speckbacher}
\address{Michael Speckbacher, Acoustics Research Institute\\
Austrian Academy of Sciences\\
Dominikanerbastei 16, 1010 Vienna, Austria}
\email{michael.speckbacher@oeaw.ac.at}
\date{}

\begin{abstract}
Let $\Omega \subset \mathbb{R}$ be a bounded interval and let $PW(\Omega )$
be the corresponding Paley--Wiener space. For a measurable set $E\subset 
\mathbb{R}$ of finite measure, consider the largest possible fraction of the 
$L^{2}$-mass of a function in $PW(\Omega )$ that can lie in $E$. We prove
that this concentration is no larger than the concentration attained on an
interval of measure $\lvert E\rvert $. 

The proof has two steps. First, we establish an optimal concentration
theorem for 
trigonometric polynomials on the circle. Second, a
universality-type controlled limit of the reproducing kernel of 
trigonometric polynomials is used to transfer the circle theorem to the real
line, by expanding circles whose projection kernels are midpoint Riemann
sums for the Paley--Wiener sinc kernel.
\end{abstract}

\maketitle


\section{The optimal concentration problem in the Paley-Wiener space}

\label{sec:1}

We use the following convention for the Fourier transform 
\begin{equation*}
\widehat{f}(\xi )=\int_{\mathbb{R}}f(x)e^{-2\pi ix\xi }\,dx.
\end{equation*}%
For a bounded, measurable and connected frequency set $\Omega $, write 
\begin{equation*}
PW(\Omega )=\left\{ f\in L^{2}(\mathbb{R}):\text{supp}\widehat{f}\subset
\Omega \right\} \text{,}
\end{equation*}%
and let $P_{\Omega }$ be the orthogonal projection from $L^{2}(\mathbb{R})$
onto $PW(\Omega )$. If, for instance, $\Omega =[-\frac{1}{2},\frac{1}{2}]$,
then $P_{\Omega }$ is an operator defined by integrating in $\mathbb{R}$
against the reproducing kernel of $PW(\Omega )$, which can be explicitly
written as 
\begin{equation}
k(x-y)=\int_{-1/2}^{1/2}e^{2\pi i\xi (x-y)}\,d\xi =%
\begin{cases}
\frac{\sin \pi (t-\xi )}{\pi (t-\xi )}, & t\neq \xi \text{,} \\[6pt]
1, & t=\xi \text{.}%
\end{cases}
\label{eq:PWkernel}
\end{equation}%
For a measurable time set $E$, let $M_{E}$ denote the operator acting on $%
L^{2}(\mathbb{R})$ by multiplication by the indicatrix function of $E$,
denoted by $\mathbf{1}_{E}$. The concentration constant of $E$ relative to $%
\Omega $ is 
\begin{equation}
\mathfrak{C}_{\Omega }(E)=\sup_{0\neq f\in PW(\Omega )}\frac{\int_{E}\lvert
f(x)\rvert ^{2}\,dx}{\int_{\mathbb{R}}\lvert f(x)\rvert ^{2}\,dx}.
\label{eq:concentration-constant}
\end{equation}%
Equivalently, $\mathfrak{C}_{W}(E)$ is the operator norm of $M_{E}P_{\Omega
}M_{E}$ and 
\begin{equation}
\mathfrak{C}_{\Omega }(E)=\lVert M_{E}P_{\Omega }\rVert ^{2}=\lVert
P_{\Omega }M_{E}\rVert ^{2}=\lVert M_{E}P_{\Omega }M_{E}\rVert .
\label{eq:operator-form}
\end{equation}%
The two positive operators 
\begin{equation*}
P_{\Omega }M_{E}P_{\Omega }|_{PW(\Omega )}\qquad \text{and}\qquad
M_{E}P_{\Omega }M_{E}|_{L^{2}(E)}
\end{equation*}%
have the same nonzero spectrum: they are respectively $R_{E}^{\ast }R_{E}$
and $R_{E}R_{E}^{\ast }$, where $R_{E}:PW(\Omega )\rightarrow L^{2}(E)$ is
the restriction operator. Thus the problem is an operator-norm extremal
problem, or, what is the same, a Rayleigh-quotient problem for the positive
concentration operator $M_{E}P_{\Omega }M_{E}$.

\begin{theorem}[Optimal Paley--Wiener concentration]
\label{thm:main} Let $\Omega \subset \mathbb{R}$ be a bounded interval of
positive length. Let $E\subset \mathbb{R}$ be measurable with $0<\lvert
E\rvert <\infty $, and let $I$ be any interval with $\lvert I\rvert =\lvert
E\rvert $. Then 
\begin{equation}
\mathfrak{C}_{\Omega }(E)\leq \mathfrak{C}_{\Omega }(I).
\label{eq:main-concentration}
\end{equation}%
Equivalently, 
\begin{equation*}
\lVert P_{\Omega }M_{E}\rVert \leq \lVert P_{\Omega }M_{I}\rVert .
\end{equation*}
\end{theorem}

This confirms the interval extremal conjecture formulated by Donoho and
Stark in 1989 \cite[Conjecture 1]{DonohoStark1989}. Their subsequent
rearrangement argument proves the result under the restriction $\lvert
\Omega \rvert \lvert E\rvert \leq 0.8$ \cite{DonohoStark1993}. Baeza--Guasch
extended this range to $\lvert \Omega \rvert \lvert E\rvert \leq 1$ \cite%
{BaezaGuasch2023}. Theorem~\ref{thm:main} removes all restrictions on the
product $\lvert \Omega \rvert \lvert E\rvert $.

Other recent related advances were triggered by the authors's reformulation
of Donoho-Stark conjecture for Daubechies localization operators \cite%
{Daubechies} with Gaussian windowed short-time Fourier transforms
(equivalent to a shape optimization problem in the Fock space of entire
functions \cite{Seip0}, optimized by Euclidean disks), in \cite[Conjecture 1]%
{AbrSpeck}. This conjecture was first proved by Galbis for radial, bounded
and integrable symbols \cite{Galbis}, and then solved completely by Nicola
and Tilli in the breakthrough paper \cite{NicTill}, using an ingenious
combination of entire function theory, measure theoretical arguments (in
particular, symmetric decreasing rearrangements) and the isoperimetrical
inequality. Shortly after, a previous observation by the first named author
and D\"{o}rfler \cite{AbreDorf}, suggesting that Daubechies-Paul
localization operators for the wavelet transform with Cauchy wavelets
(equivalent to concentration operators in Bergman spaces \cite{Seip0}) are
optimized by pseudohyperbolic disks, has been confirmed by Ramos and Tilli 
\cite{RamTilli}, adapting the methods of \cite{NicTill} to the hyperbolic
setting. Then, Frank \cite{Frank} extended the results to general coherent
states and solved the concentration problem for analytic polynomials in the
sphere, where spherical caps provide the optimal shape. The problem in the
Hardy space has been solved by G\'{o}mez, Kalaj, Melentijevi\'{c} and Ramos
in \cite{ProcLMS}, using a controlled limiting process, where the scale of
Bergman spaces optimal concentrations approaches the lower endpoint, in a
work whose main goal was the extension to wavelets of the stability results
of G\'{o}mez, Guerra, Ramos and Tilli \cite{GuerraGomesTilliRamos}. Similar
stability problems have been considered by Garc\'{\i}a-Ferrero and
Ortega-Cerd\'{a} for analytic polynomials in the sphere \cite{CIMP}.

Despite this impressive body of work, the original Donoho-Stark conjecture
remained open and cannot be solved using the above described circle of
ideas. The proof offered in this manuscript follows a completely new
direction and consists of two steps: first we solve a concentration problem
for \emph{analytic trigonometric polynomials on the circle} (Theorem \ref%
{thm:circle}, which has independent interest) and then take the controlled
limit to yield the result in the Paley-Wiener space. The controlled limiting
process that proves Theorem \ref{thm:main} (see section \ref{sec:11} for
details) is based on the universality-type limit (see \cite{Lubinsky} for
limits of this nature) connecting the reproducing kernel of analytic
trigonometric polynomials on the circle with the reproducing kernel of the
Paley-Wiener space (\ref{eq:PWkernel}), 
\begin{equation*}
\lim_{N\rightarrow \infty }\frac{1}{2N+1}K_{N}(e^{i\frac{\pi t}{N}},e^{i%
\frac{\pi \xi }{N}})=\frac{\sin \pi (t-\xi )}{\pi (t-\xi )}\text{, \ \ \
where \ \ \ \ }K_{N}(z,w)=\sum\limits_{j=-N}^{N}\left( z\overline{w}\right)
^{j}\text{,}
\end{equation*}%
which, for the problem at hand, will be taken by simultaneous expanding
large-degrees and large-circles. The methods of \cite{NicTill} have proved
to be extremely flexible, but they cannot be adapted to the geometry of the
Paley-Wiener space and also fail to work for analytic polynomials in the
circle for at least two reasons. First, the curvature term $\Delta \log
K_{N}^{+}(z,z)$ is radial but not constant, and this is essential in \cite%
{NicTill}. Second,\ the geometric optimization provided by the variations of
the isoperimetric inequality does not seem to have a direct analogue,
neither in the circle, nor in the geometric version provided by the
Beurling-Lax-Riesz decomposition, $L^{2}(\mathbb{T})=H(\mathbb{D})\oplus H(%
\mathbb{D}^{c})$, where the ambient space is the double-sheet space $\mathbb{%
D}\bigsqcup \mathbb{D}^{c}$.

\textbf{Origins of the proof and the use of AI}

A paragraph aimed at clarifying the sources of the leading ideas in this
paper, including an interesting saga of interactions between human
researchers and language models is in order. The leading idea of approaching
the problem by first proving the circle Theorem~\ref{thm:circle} and then
using a controlled limit process, is due to the authors and has been
partially inspired by the controlled limiting process used to solve the
shape optimization problem in the Hardy space \cite{ProcLMS}. To prove
Theorem~\ref{thm:circle} we had considerable help from modern AI tools,
primarily GPT-5.6 Sol and GPT-5.5 Pro. In fact, the core technical
innovation in the proof, that extended the result from $N=2$ (second degree
polynomials in the circle, which is the only easy non-trivial case) to $N=3$
and subsequently to general $N$, was found during long exploratory
conversations that started with GPT 5.5 Pro and were continued with GPT-5.6
Sol. This preliminary version still keeps part of the interesting graphic
language used by the Language Model.

\section{The optimal concentration problem for analytic polynomials in the
circle}

\label{sec:2}

\subsection{The concentration operator}

Let 
\begin{equation*}
\mathbb{T}=\mathbb{R}/(2\pi \mathbb{Z}),\qquad \mathcal{A}_{N}=\text{span}_{%
\mathbb{C}}\{1,e^{it},\ldots ,e^{iNt}\}.
\end{equation*}%
The circle carries angular measure $dt$, so that $\lvert \mathbb{T}\rvert
=2\pi $. Let $P_{N}$ be the orthogonal projection from $L^{2}(\mathbb{T})$
onto $\mathcal{A}_{N}$. For a measurable set $E\subset \mathbb{T}$, define 
\begin{equation*}
T_{E}^{(N)}=\left. P_{N}M_{E}P_{N}\right\vert _{\mathcal{A}_{N}}\text{.}
\end{equation*}%
Its largest eigenvalue is 
\begin{equation}
\lambda _{\max }\left( T_{E}^{(N)}\right) =\sup_{0\neq p\in \mathcal{A}_{N}}%
\frac{\int_{E}\lvert p(t)\rvert ^{2}\,dt}{\int_{\mathbb{T}}\lvert p(t)\rvert
^{2}\,dt}\text{.}  \label{eq:circle-rayleigh}
\end{equation}
Theorem~\ref{thm:main} will follow as a limit case of the following result
on the circle, which has independent interest.

\begin{theorem}[Circle interval theorem]
\label{thm:circle} Let $N\geq 0$. If $E\subset \mathbb{T}$ is measurable and 
$I\subset \mathbb{T}$ is a circle interval with $\lvert I\rvert =\lvert
E\rvert $, then 
\begin{equation*}
\lambda _{\max }\left( T_{E}^{(N)}\right) \leq \lambda _{\max }\left(
T_{I}^{(N)}\right) \text{.}
\end{equation*}
\end{theorem}

The case $N=0$, as well as the cases $\lvert E\rvert =0$ and $\lvert E\rvert
=2\pi $, are immediate. Therefore, assume from now on that 
\begin{equation*}
N\geq 1,\qquad 0<\lvert E\rvert <2\pi \text{.}
\end{equation*}%
We now observe that taking complements turns the maximization into a
minimization problem. Set $F=\mathbb{T}\setminus E$. Then 
\begin{equation*}
T_{E}^{(N)}+T_{F}^{(N)}=I_{\mathcal{A}_{N}}\text{,}
\end{equation*}%
gives 
\begin{equation}
\lambda _{\max }\left( T_{E}^{(N)}\right) =1-\lambda _{\min }\left(
T_{F}^{(N)}\right) \text{.}  \label{eq:complement-eigenvalue}
\end{equation}%
For $0<\mu <2\pi $, define 
\begin{equation}
\Lambda _{N}(\mu )=\inf_{\substack{ F\subset \mathbb{T},\ \lvert F\rvert
=\mu  \\ 0\neq p\in \mathcal{A}_{N}}}\frac{\int_{F}\lvert p(t)\rvert ^{2}dt}{%
\int_{\mathbb{T}}\lvert p(t)\rvert ^{2}dt}\text{.}  \label{eq:Lambda}
\end{equation}%
It is enough to prove that this joint infimum is attained when $F$ is a
circle interval.

\subsection{A variational problem}

\label{sec:3}

Let $\mathcal{D}_{N}$ be the set of normalized nonnegative trigonometric
polynomials 
\begin{equation*}
\mathcal{D}_{N}=\left\{ g(t)=\sum_{k=-N}^{N}c_{k}e^{ikt}:g\geq 0,\quad
c_{-k}=\overline{c_{k}},\quad \int_{\mathbb{T}}g=1\right\} \text{.}
\end{equation*}%
We shall use the scalar Fej\'{e}r--Riesz theorem.

\begin{theorem}[Fej\'er--Riesz factorization]
\label{thm:fejer-riesz} If $g$ is a nonnegative trigonometric polynomial of
degree at most $N$, then there exists an analytic polynomial $p\in\mathcal{A}%
_N$ such that 
\begin{equation*}
g(t)=|p(e^{it})|^2.
\end{equation*}
\end{theorem}

This classical factorization is discussed, for example, in \cite%
{DritschelRovnyak}. By Theorem~\ref{thm:fejer-riesz}, these are exactly the
normalized energy densities 
\begin{equation*}
g(t)=\frac{\lvert p(t)\rvert ^{2}}{\int_{\mathbb{T}}\lvert p(t)\rvert ^{2}dt}%
,\qquad 0\neq p\in \mathcal{A}_{N}\text{.}
\end{equation*}%
For $g\in \mathcal{D}_{N}$, define its lower-tail functional by 
\begin{equation}
L_{\mu }(g)=\min_{\lvert F\rvert =\mu }\int_{F}g(t)dt\text{.}
\label{eq:lower-tail}
\end{equation}%
Then 
\begin{equation}
\Lambda _{N}(\mu )=\min_{g\in \mathcal{D}_{N}}L_{\mu }(g).
\label{eq:Lambda-tail}
\end{equation}
For a fixed density, the minimizing set can be obtained from the following version
of the bathtub principle. We provide a proof for completeness.

\begin{lemma}[Bathtub principle]
\label{lem:bathtub} Let $g$ be continuous on $\mathbb{T}$. For some $\tau\in%
\mathbb{R}$, a minimizing set in \eqref{eq:lower-tail} can be chosen so that 
\begin{equation*}
\{g<\tau\}\subset F\subset\{g\leq\tau\}, \qquad \lvert F\rvert=\mu.
\end{equation*}
Thus one minimizes the integral by selecting the $\mu$ smallest values of $g$%
.
\end{lemma}

\begin{proof}
Choose $\tau $ so that 
\begin{equation*}
\lvert \{g<\tau \}\rvert \leq \mu \leq \lvert \{g\leq \tau \}\rvert ,
\end{equation*}%
and fill the required portion of the level set $\{g=\tau \}$. If $A$ is any
other set of measure $\mu $, then 
\begin{equation*}
\int_{A}g-\int_{F}g=\int_{A\setminus F}(g-\tau )+\int_{F\setminus A}(\tau
-g)\geq 0.
\end{equation*}
The first equality holds because $|A\backslash F|-|F\backslash A|=|A|-|A\cap
F|-(|F|-|A\cap F|)=0$.
\end{proof}

The existence of a global minimizer follows from a compactness and concavity
argument.

\begin{lemma}
\label{lem:compact-concave} The set $\mathcal{D}_N$ is compact in its
coefficient space. The map $\mathscr L_\mu:\mathcal{D}_N\to\mathbb{R}$ is
continuous and concave. Consequently it has a global minimizer.
\end{lemma}

\begin{proof}
For $g\in \mathcal{D}_{N}$, 
\begin{equation*}
c_{k}=\frac{1}{2\pi }\int_{\mathbb{T}}g(t)e^{-ikt}\,dt,\qquad \lvert
c_{k}\rvert \leq \frac{1}{2\pi }.
\end{equation*}%
The defining conditions of $\mathcal{D}_{N}$ are closed, so
finite-dimensional compactness follows. Moreover, 
\begin{equation*}
\lvert \mathscr L_{\mu }(g)-\mathscr L_{\mu }(h)\rvert \leq \int_{\mathbb{T}%
}\lvert g-h\rvert ,
\end{equation*}%
which gives continuity. Finally, for $0\leq s\leq 1$, 
\begin{equation*}
\mathscr L_{\mu }(sg+(1-s)h)\geq s\mathscr L_{\mu }(g)+(1-s)\mathscr L_{\mu
}(h),
\end{equation*}%
because $\mathscr L_{\mu }$ is the infimum of the linear functionals $%
g\mapsto \int_{F}g$ over sets of measure $\mu $.
\end{proof}

\section{Zero saturation of an extremal density}

\label{sec:4}

\subsection{Selecting one extremizer}

For a nonzero nonnegative trigonometric polynomial $g$, define 
\begin{equation*}
Z(g)=\frac{1}{2}\sum_{\theta :g(\theta )=0}\text{ord}_{\theta }g.
\end{equation*}%
Every zero has even order, and $Z(g)\in \{0,\ldots ,N\}$. Among all global
minimizers of $\mathscr L_{\mu }$, choose one for which $Z(g)$ is maximal.
We will use only this minimizer; no uniqueness is claimed or needed. Let its
distinct zeros be $\theta _{1},\ldots ,\theta _{r}$, with respective orders $%
2\nu _{1},\ldots ,2\nu _{r}$, and put 
\begin{equation*}
M=\nu _{1}+\cdots +\nu _{r}.
\end{equation*}%
Since $z^{N}g(t)$, with $z=e^{it}$, is an ordinary polynomial of degree at
most $2N$, we have $M\leq N$.

\subsection{Factoring out the zero set}

\begin{lemma}
\label{lem:zero-factor} Define 
\begin{equation*}
w(t) = \prod_{j=1}^{r} \left(2\sin\frac{t-\theta_j}{2}\right)^{2\nu_j}.
\end{equation*}
Then 
\begin{equation}  \label{eq:g-weta}
g(t)=w(t)\eta(t),
\end{equation}
where $\eta$ is a real trigonometric polynomial of degree at most $N-M$ and $%
\eta(t)>0$ for every $t\in\mathbb{T}$.
\end{lemma}

\begin{proof}
The polynomial $w$ has degree $M$ and exactly the same zeros, with the same
orders, as $g$. For a nonzero constant, 
\begin{equation*}
w(t)=Cz^{-M}\prod_{j=1}^{r}(z-e^{i\theta _{j}})^{2\nu _{j}}.
\end{equation*}%
If $G(z)=z^{N}g(t)$, then $G$ is divisible by 
\begin{equation*}
W(z)=\prod_{j=1}^{r}(z-e^{i\theta _{j}})^{2\nu _{j}}.
\end{equation*}%
Consequently, 
\begin{equation*}
\eta (t)=z^{M-N}\frac{G(z)}{C W(z)}
\end{equation*}%
is a real trigonometric polynomial of degree at most $N-M$. Away from the
zeros, $\eta =g/w>0$, as $g,w>0$. At a zero, the quotient has a removable
singularity and extends to a positive value because $g$ and $w$ have the
same exact even order there.
\end{proof}

\subsection{A two-sided variation argument}

\begin{lemma}
\label{lem:max-zeros} For the selected minimizer, $M=N$.
\end{lemma}

\begin{proof}
Suppose that $M<N$. Set 
\begin{equation*}
\beta =\frac{\int_{\mathbb{T}}w(t)\cos t\,dt}{\int_{\mathbb{T}}w(t)\,dt}%
,\qquad \psi (t)=\cos t-\beta .
\end{equation*}%
The function $w$ is continuous and positive on a nonempty open set. Its
weighted average of $\cos t$ therefore cannot equal either endpoint $\pm 1$.
Hence 
\begin{equation*}
-1<\beta <1,\qquad \psi \not\equiv 0,\qquad \int_{\mathbb{T}}w\psi =0.
\end{equation*}%
Also $\deg (w\psi )\leq M+1\leq N$. Since $\eta >0$, define 
\begin{equation*}
\varepsilon _{\ast }=\left( \max_{t\in \mathbb{T}}\frac{\lvert \psi
(t)\rvert }{\eta (t)}\right) ^{-1}>0,
\end{equation*}%
and 
\begin{equation*}
g_{\pm }=w(\eta \pm \varepsilon _{\ast }\psi ).
\end{equation*}%
The choice of $\varepsilon _{\ast }$ gives 
\begin{equation*}
g_{\pm }\geq 0,\qquad \deg g_{\pm }\leq N,\qquad \int_{\mathbb{T}}g_{\pm
}=1,\qquad g=\frac{g_{+}+g_{-}}{2}.
\end{equation*}%
Thus $g_{\pm }\in \mathcal{D}_{N}$. Let $m_{0}=\min_{\mathcal{D}_{N}}%
\mathscr
L_{\mu }$. Concavity and minimality give 
\begin{equation*}
m_{0}=\mathscr L_{\mu }(g)\geq \frac{\mathscr L_{\mu }(g_{+})+\mathscr %
L_{\mu }(g_{-})}{2}\geq m_{0}.
\end{equation*}%
Each of the two terms in the middle is at least $m_{0}$, so both are equal
to $m_{0}$. Hence $g_{+}$ and $g_{-}$ are also global minimizers. At a point
where $\lvert \psi \rvert /\eta $ is maximal, one of the factors $\eta \pm
\varepsilon _{\ast }\psi $ touches zero. The touching factor cannot vanish
on a nontrivial arc: analyticity would make it identically zero, but since $%
\psi$ has a zero on $\mathbb{T} $ and $\eta>0$ this cannot happen. It is
therefore a nonzero nonnegative real-analytic function with an isolated zero
of positive even order. This either creates a new zero of $g_{\pm }$ or
increases the order of an old zero of $w$. In both cases one of the two new
minimizers has larger $Z$ than $g$, contradicting the choice of $g$. Thus $%
M=N$ and the proof is concluded.
\end{proof}

\begin{remark}
The argument is not an iterative algorithm and does not require a canonical
extremizer. If one unused degree existed, a single introduction of a zero would
contradict maximality of the integer $Z(g)$.
\end{remark}

Since $M=N$, the remainder $\eta$ in \eqref{eq:g-weta} is a positive
constant. There is therefore a real amplitude 
\begin{equation}  \label{eq:q-product}
q(t) = \prod_{j=1}^{r} \left(2\sin\frac{t-\theta_j}{2}\right)^{\nu_j},
\qquad \sum_{j=1}^{r}\nu_j=N,
\end{equation}
such that 
\begin{equation}  \label{eq:g-q2}
g(t)=\frac{q(t)^2}{D}, \qquad D=\int_\mathbb{T }q(t)^2\,dt.
\end{equation}
Moreover, 
\begin{equation*}
e^{iNt/2}q(t)\in\mathcal{A}_N.
\end{equation*}
For odd $N$, the amplitude $q$ is antiperiodic. Its square, its logarithmic
derivative, and every product of two amplitudes used below are nevertheless $%
2\pi$-periodic. Repeated roots are always counted according to their
multiplicities.

\section{Geometry of the lower-level set}

\label{sec:5}

\subsection{Monotonicity between consecutive roots
}

Away from the roots of $q$, 
\begin{equation}  \label{eq:logderivative}
\frac{q^{\prime }(t)}{q(t)} = \frac12\sum_{j=1}^{r}\nu_j \cot\frac{t-\theta_j%
}{2},
\end{equation}
and hence 
\begin{equation}  \label{eq:logderivative-decreasing}
\left(\frac{q^{\prime }}q\right)^{\prime }(t) = -\frac14\sum_{j=1}^{r}\nu_j
\csc^2\frac{t-\theta_j}{2} <0.
\end{equation}

\begin{lemma}[One-hill lemma]
\label{lem:one-hill} Between two consecutive distinct roots of $q$, the
function $\lvert q\rvert$ increases strictly from zero to one maximum and
then decreases strictly to zero.
\end{lemma}

\begin{proof}
Lift the root gap to $(\alpha ,\gamma )\subset \mathbb{R}$. The function $q$
has constant sign there and 
\begin{equation*}
\frac{d}{dt}\log \lvert q(t)\rvert =\frac{q^{\prime }(t)}{q(t)}.
\end{equation*}%
By \eqref{eq:logderivative-decreasing}, this logarithmic derivative is
strictly decreasing. The zero at $\alpha $ forces the limit $+\infty $ at
the left endpoint, while the zero at $\gamma $ forces the limit $-\infty $
at the right endpoint. Thus it vanishes exactly once. The cyclic gap is
treated by lifting the first root to $\theta _{1}+2\pi $.
\end{proof}

\subsection{Components and root blocks}

By the bathtub principle, the minimizing set for $q^2/D$ is a lower level
set. A nonconstant trigonometric polynomial has only finitely many points on
any fixed level. Therefore the distribution function of $q^2$ is continuous,
and there exists $\rho>0$ such that 
\begin{equation}  \label{eq:Fstar}
F^* = \{t\in\mathbb{T}:q(t)^2\leq\rho^2\}, \qquad \lvert F^*\rvert=\mu.
\end{equation}
The positivity of $\rho$ follows from $\mu>0$, since the zero set of $q$ is
finite.

The one-hill lemma gives a complete description in each root gap:

\begin{itemize}
\item If the hilltop is higher than $\rho$, the equation $q^2=\rho^2$ has
exactly two transverse solutions in the gap.

\item If the hilltop equals $\rho$, the entire gap belongs to $F^*$; the
hilltop is an internal point joining two low pieces.

\item If the hilltop is lower than $\rho$, the entire gap belongs to $F^*$.
\end{itemize}
Since $\mu <2\pi $, at least one gap has hilltop strictly above $\rho $.
Choose the cut point $x_{0}$ in the open superlevel set of a high root gap,
for instance at its hilltop, so that $q(x_{0})^2>\rho ^{2}$. Then no component
of $F^{\ast }$ is split by the cut. On the resulting lift, write 
\begin{equation}
F^{\ast }=\bigcup_{\ell =1}^{m}[u_{\ell },v_{\ell }]  \label{eq:F-components}
\end{equation}%
as a union of disjoint, correctly ordered intervals. Every component
contains a nonempty consecutive block $B_{\ell }$ of roots, with
multiplicity. Only the two outer endpoints move in the variational argument.
They are transverse and satisfy 
\begin{equation}
(q^{2})^{\prime }(u_{\ell })<0,\qquad (q^{2})^{\prime }(v_{\ell })>0\text{.}
\label{eq:endpoint-orientation}
\end{equation}%
Hilltops at the threshold are internal points and do not enter the endpoint
differentiation.

\section{Stationarity for the fixed lower-level set}

\label{sec:6}

Introduce the real amplitude space 
\begin{equation*}
\mathcal{V}_N = \left\{ h:\mathbb{R}\to\mathbb{R}: e^{iNt/2}h(t)\in\mathcal{A%
}_N \right\}.
\end{equation*}
Here and below, the condition $e^{iNt/2}h(t)\in\mathcal{A}_N$ means that the
function 
\begin{equation*}
t\longmapsto e^{iNt/2}h(t)
\end{equation*}
is the $2\pi$-periodic lift of an element of $\mathcal{A}_N$. Periodicity of
this lift gives 
\begin{equation*}
h(t+2\pi)=(-1)^N h(t).
\end{equation*}
Therefore, elements of $\mathcal{V}_N$ are periodic when $N$ is even and
antiperiodic when $N$ is odd. In either case, the product of two elements of 
$\mathcal{V}_N$ is $2\pi$-periodic. Consequently every product integrated
below defines an unambiguous function on $\mathbb{T}$.

For $0\neq h\in\mathcal{V}_N$, put 
\begin{equation*}
\mathscr Q_\mu(h) = \min_{\lvert A\rvert=\mu} \frac{\displaystyle\int_Ah^2} {%
\displaystyle\int_\mathbb{T }h^2}.
\end{equation*}
The minimum is over measurable sets $A\subset \mathbb{T}$ with $\lvert
A\rvert =\mu $. The denominator is positive because $h\neq 0$. Moreover, if 
\begin{equation*}
p_{h}(t)=e^{iNt/2}h(t),
\end{equation*}%
then $p_{h}\in \mathcal{A}_{N}$ and 
\begin{equation*}
\lvert p_{h}(t)\rvert ^{2}=h(t)^{2}.
\end{equation*}%
It follows from the description of $\mathcal{D}_{N}$ by normalized energy
densities that 
\begin{equation*}
g_{h}(t):=\frac{h(t)^{2}}{\displaystyle\int_{\mathbb{T}}h(t)^{2}dt}=\frac{%
\lvert p_{h}(t)\rvert ^{2}}{\displaystyle\int_{\mathbb{T}}\lvert
p_{h}(t)\rvert ^{2}dt}\in \mathcal{D}_{N}.
\end{equation*}%
This is the precise link between the real amplitude problem on $\mathcal{V}%
_{N}$ and the density problem on $\mathcal{D}_{N}$.

Let $F=F^{\ast }$. We now justify each part of the comparison that follows.
Since $\lvert F\rvert =\mu $, the set $F$ is one of the competitors in the
definition of $\mathscr Q_{\mu }(h)$. Hence 
\begin{equation*}
\frac{\displaystyle\int_{F}h^{2}}{\displaystyle\int_{\mathbb{T}}h^{2}}\geq %
\mathscr Q_{\mu }(h).
\end{equation*}%
The density $g_{h}$ belongs to $\mathcal{D}_{N}$, while $q^{2}/\int_{\mathbb{%
T}}q^{2}$ is a global minimizer of the lower-tail functional over $\mathcal{D%
}_{N}$. Therefore 
\begin{equation*}
\mathscr Q_{\mu }(h)\geq \mathscr Q_{\mu }(q).
\end{equation*}%
Finally, $F=F^{\ast }$ is a lower-level set selected by the bathtub
principle for $q^{2}$, so it attains the minimum defining $\mathscr Q_{\mu
}(q)$. Combining these three observations yields, for every $0\neq h\in 
\mathcal{V}_{N}$, 

\begin{equation}
\frac{\int_{F}h^{2}}{\int_{\mathbb{T}}h^{2}}\geq \mathscr Q_{\mu }(h)\geq %
\mathscr Q_{\mu }(q)=\frac{\int_{F}q^{2}}{\int_{\mathbb{T}}q^{2}}.
\label{eq:fixed-set-chain}
\end{equation}%
Thus $q$ minimizes not only the lower-tail problem, but also the Rayleigh
quotient for the now fixed set $F$. Write 
\begin{equation}
J=\int_{F}q^{2},\qquad D=\int_{\mathbb{T}}q^{2},\qquad \lambda =\frac{J}{D},
\label{eq:J-D-lambda}
\end{equation}%
and define the symmetric bilinear form 
\begin{equation}
B(h,k)=\int_{F}hk-\lambda \int_{\mathbb{T}}hk.  \label{eq:B-form}
\end{equation}%
Because $q\neq0$, we have $D>0$. Fix a real direction $h\in\mathcal{V}_N$
and, for real $s$, define 
\begin{equation*}
R_h(s) := \frac{\displaystyle\int_F(q+s h)^2} {%
\displaystyle\int_\mathbb{T}(q+s h)^2}.
\end{equation*}
The space $\mathcal{V}_N$ is a real vector space, so $q+s h\in%
\mathcal{V}_N$. The denominator of $R_h$ equals $D$ at $s=0$ and
depends continuously on $s$; hence it remains positive for all
sufficiently small $\lvert s\rvert$. Equation %
\eqref{eq:fixed-set-chain} shows that $R_h$ has a two-sided minimum at $%
\varepsilon=0$. Consequently $R_h^{\prime }(0)=0$.

For completeness, differentiating the numerator and denominator separately
gives 
\begin{equation*}
\left. \frac{d}{ds }\right\vert _{s
=0}\int_{F}(q+s h)^{2}=2\int_{F}qh,\qquad \left. \frac{d}{%
ds}\right\vert _{s =0}\int_{\mathbb{T}}(q+s
h)^{2}=2\int_{\mathbb{T}}qh.
\end{equation*}%
The quotient rule and $J=\lambda D$ therefore yield 
\begin{align*}
R_{h}^{\prime }(0)& =\frac{2D\displaystyle\int_{F}qh-2J\displaystyle\int_{%
\mathbb{T}}qh}{D^{2}} \\
& =\frac{2}{D}\left( \int_{F}qh-\lambda \int_{\mathbb{T}}qh\right) =\frac{2}{%
D}B(q,h).
\end{align*}%
Since $D>0$ and $R_{h}^{\prime }(0)=0$, we obtain

\begin{equation}
B(q,h)=0\qquad \text{for every }h\in \mathcal{V}_{N}.
\label{eq:stationarity}
\end{equation}%
This is stationarity on the real vector space $\mathcal{V}_{N}$. No
moving-set differentiation is involved here: throughout this calculation the
set $F$ is fixed, and only the amplitude is varied along the affine line $%
q+s h$.

\begin{remark}
Equation \eqref{eq:fixed-set-chain} is the bridge between the set-dependent
lower-tail problem and ordinary finite-dimensional stationarity. Later the computation of the 
Hessian  uses \eqref{eq:stationarity} repeatedly.
\end{remark}

\section{Relative translations of the root blocks}

\label{sec:7}

For $0\neq h\in\mathcal{V}_N$, put 
\begin{equation*}
\mathscr Q_\mu(h) = \min_{\lvert A\rvert=\mu} \frac{\displaystyle\int_Ah^2} {%
\displaystyle\int_\mathbb{T }h^2}.
\end{equation*}
The minimum is over measurable sets $A\subset\mathbb{T}$ with $\lvert
A\rvert=\mu$. The denominator is positive because $h\neq0$. Moreover, if 
\begin{equation*}
p_h(t)=e^{iNt/2}h(t),
\end{equation*}
then $p_h\in\mathcal{A}_N$ and 
\begin{equation*}
\lvert p_h(t)\rvert^2=h(t)^2.
\end{equation*}
It follows from the description of $\mathcal{D}_N$ by normalized energy
densities that 
\begin{equation*}
g_h(t) := \frac{h(t)^2}{\displaystyle\int_\mathbb{T }h^2} = \frac{\lvert
p_h(t)\rvert^2} {\displaystyle\int_\mathbb{T}\lvert p_h\rvert^2} \in\mathcal{%
D}_N.
\end{equation*}
This is the precise link between the real amplitude problem on $\mathcal{V}%
_N $ and the density problem on $\mathcal{D}_N$.

\begin{itemize}
\item 
We will now assume for contradiction that $m \geq 2$ in \eqref{eq:F-components}. This allows at least two rootblocks to be translated relative to one another. We retain this assumption throughout the following sections and eventually show, in \eqref{neg:def}, that every nonconstant relative block translation has strictly negative second variation. This contradicts the local minimality established in \eqref{eq:R-local-min} and therefore proves that $m = 1$.
\end{itemize}

Because $\rho>0$, no root of $q$ lies on the boundary of $F^*$. Every
distinct root therefore belongs to the interior of exactly one component of $%
F^*$. For $1\leq\ell\leq m$, let 
\begin{equation*}
I_\ell = \{j\in\{1,\ldots,r\}:\theta_j\in B_\ell\}.
\end{equation*}
Then $I_1,\ldots,I_m$ partitions $\{1,\ldots,r\}$, and each $I_\ell$ is
nonempty and consecutive in the fixed lift chosen in Section~\ref{sec:5}.
The multiplicity of the distinct root $\theta_j$ remains the exponent $\nu_j$%
; it is not included a second time in the index set.

Factor $q $ according to the component blocks: 
\begin{equation}
Q_{\ell }(t)=\prod_{\theta _{j}\in B_{\ell }}\left( 2\sin \frac{t-\theta _{j}%
}{2}\right) ^{\nu _{j}},\qquad q(t)=C\prod_{\ell =1}^{m}Q_{\ell }(t).
\label{eq:block-factorization}
\end{equation}%
Translate the blocks independently: 
\begin{equation}
q_{c}(t)=C\prod_{\ell =1}^{m}Q_{\ell }(t-c_{\ell }),\qquad c=(c_{1},\ldots
,c_{m})\in \mathbb{R}^{m}.  \label{eq:block-translation}
\end{equation}%
We verify explicitly that these perturbations remain in the amplitude space.
For $a\in\mathbb{R}$, writing $z=e^{it}$, one has 
\begin{equation*}
e^{it/2}\,2\sin\frac{t-a}{2} = -i e^{-ia/2}\bigl(z-e^{ia}\bigr).
\end{equation*}
Apply this identity to $a=\theta_j+c_\ell$ for $j\in I_\ell$. Since $%
\sum_{j=1}^{r}\nu_j=N$, multiplication of all the factors gives 
\begin{equation*}
e^{iNt/2}q_c(t)\in\mathcal{A}_N.
\end{equation*}
Thus $q_c\in\mathcal{V}_N$ for every $c\in\mathbb{R}^m$. The product is not
identically zero, and hence 
\begin{equation*}
\int_\mathbb{T }q_c(t)^2\,dt>0.
\end{equation*}
For odd $N$, the amplitude $q_c$ is antiperiodic, while $q_c^2$ is $2\pi$%
-periodic. All quotients and circle integrals below are therefore well
defined.


We work throughout on the fixed real lift obtained from the cut in Section~%
\ref{sec:5}; every endpoint below refers to this lift. At every outer
endpoint $e$, define 
\begin{equation*}
\sigma _{e}=\frac{q(e)}{\rho }\in \{-1,1\}.
\end{equation*}%
Since $\rho >0$ and $(q^{2})^{\prime }(e)\neq 0$, we also have $q^{\prime
}(e)\neq 0$. For later reference, the endpoint orientations imply 
\begin{equation*}
\sigma _{u_{\ell }}q^{\prime }(u_{\ell })<0,\qquad \sigma _{v_{\ell
}}q^{\prime }(v_{\ell })>0.
\end{equation*}%
For an outer endpoint $e$, set 
\begin{equation*}
H_{e}(x,c,\tau )=q_{c}(x)-\sigma _{e}\tau .
\end{equation*}%
This is a real analytic function satisfying 
\begin{equation*}
H_{e}(e,0,\rho )=0,\qquad \text{and}\qquad \partial _{x}H_{e}(e,0,\rho
)=q^{\prime }(e)\neq 0.
\end{equation*}%
The analytic implicit-function theorem therefore supplies a unique
real-analytic endpoint branch $e(c,\tau )$ satisfying 
\begin{equation}
e(0,\rho )=e,\qquad \text{and}\qquad H_{e}(e(c,\tau ),c,\tau )=0.
\label{eq:endpoint}
\end{equation}%
There are only finitely many outer endpoints, so we may intersect the
resulting neighborhoods and work on one common neighborhood of $(0,\rho )$
on which all functions $u_{\ell }(c,\tau )$ and $v_{\ell }(c,\tau )$ are
defined. In particular, all their second derivatives exist. Differentiating
the two signed crossing equations \eqref{eq:endpoint} with respect to $\tau $
at $(0,\rho )$ and using the definition of $H_{e}$ gives 
\begin{equation*}
q^{\prime }(u_{\ell })\,\partial _{\tau }u_{\ell }(0,\rho )=\sigma _{u_{\ell
}},\qquad q^{\prime }(v_{\ell })\,\partial _{\tau }v_{\ell }(0,\rho )=\sigma
_{v_{\ell }}.
\end{equation*}%
Using the endpoint orientations above, this becomes

\begin{equation*}
\partial _{\tau }u_{\ell }(0,\rho )=-\frac{1}{\lvert q^{\prime }(u_{\ell
})\rvert },\qquad \partial _{\tau }v_{\ell }(0,\rho )=\frac{1}{\lvert
q^{\prime }(v_{\ell })\rvert }.
\end{equation*}%
Set 
\begin{equation*}
\Phi (c,\tau )=\sum_{\ell =1}^{m}\left( v_{\ell }(c,\tau )-u_{\ell }(c,\tau
)\right) \text{.}
\end{equation*}%
Then 
\begin{equation*}
\partial _{\tau }\Phi (0,\rho )=\sum_{e}\frac{1}{\lvert q^{\prime }(e)\rvert 
}>0\text{.}
\end{equation*}
The analytic implicit-function theorem now gives a neighborhood $U\subset%
\mathbb{R}^m$ of $0$, a neighborhood $I\subset\mathbb{R}$ of $\rho$, and a
unique analytic function 
\begin{equation*}
\rho(\,\cdot\,):U\longrightarrow I
\end{equation*}
such that 
\begin{equation*}
\rho(0)=\rho, \qquad \text{and}\qquad\Phi(c,\rho(c))=\mu \quad(c\in U).
\end{equation*}
The uniqueness here is local, near the original level $\rho$. After
shrinking $U$, we may assume that $\rho(c)>0$, that all endpoints remain in
the fixed lift and in the same strict cyclic order, and that the endpoint
intervals remain pairwise disjoint.

Define 
\begin{equation*}
G_{c}=\bigcup_{\ell =1}^{m}[u_{\ell }(c,\rho (c)),v_{\ell }(c,\rho (c))],
\end{equation*}%
and 
\begin{equation}
\mathcal{R}(c)=\frac{\int_{G_{c}}q_{c}^{2}}{\int_{\mathbb{T}}q_{c}^{2}}.
\label{eq:moving-rayleigh}
\end{equation}
Because the intervals are disjoint on the fixed lift, 
\begin{equation*}
\lvert G_c\rvert = \sum_{\ell=1}^{m} \bigl(v_\ell(c,\rho(c))-u_\ell(c,%
\rho(c))\bigr) = \Phi(c,\rho(c)) = \mu.
\end{equation*}
The initial conditions for the endpoint branches and $\rho(0)=\rho$ give $%
G_0=F^*$, up to boundary points, which do not affect any integral. The
numerator can be written as a finite sum of integrals over analytic
endpoints, 
\begin{equation}  \label{eq:expl-integral}
\int_{G_c}q_c^2 = \sum_{\ell=1}^{m}
\int_{u_\ell(c,\rho(c))}^{v_\ell(c,\rho(c))}q_c(t)^2\,dt.
\end{equation}
Together with the positive analytic denominator, this shows that $\mathcal{R}
$ is real analytic on a possibly smaller neighborhood of $0$. In particular,
the first and second derivatives used in Section~\ref{sec:8} exist.

Indeed, admissibility of $G_c$ gives 
\begin{equation*}
\mathcal{R}(c) \geq \min_{\substack{ A\subset\mathbb{T}\ \mathrm{measurable} 
\\ \lvert A\rvert=\mu}} \frac{\displaystyle\int_Aq_c^2} {\displaystyle\int_%
\mathbb{T }q_c^2} = \mathscr Q_\mu(q_c).
\end{equation*}
Since $q_c\in\mathcal{V}_N\setminus\{0\}$ and $q$ is a global minimizer of $%
\mathscr Q_\mu$ on this space, 
\begin{equation*}
\mathscr Q_\mu(q_c)\geq\mathscr Q_\mu(q).
\end{equation*}
Finally, $G_0=F^*$, and $F^*$ is a minimizing lower-level set for $q^2$, so $%
\mathscr Q_\mu(q)=\mathcal{R}(0)$. We have therefore proved, for every
sufficiently small $c$, that

\begin{equation}
\mathcal{R}(c)\geq \mathscr Q_{\mu }(q_{c})\geq \mathscr Q_{\mu }(q)=%
\mathcal{R}(0).  \label{eq:R-local-min}
\end{equation}%
Hence $c=0$ is a local minimum of $\mathcal{R}$.

\section{Computing the Hessian
}

\label{sec:8}

\subsection{Endpoint velocities}

At an outer endpoint $e$, set 
\begin{equation*}
\kappa _{e}=\frac{1}{\lvert q^{\prime }(e)\rvert }.
\end{equation*}%
For a tangent amplitude $h\in \mathcal{V}_{N}$, define 
\begin{equation*}
z_{h}(e)=\sigma _{e}h(e),\qquad \text{and}\qquad \overline{z}_{h}=\frac{%
\sum_{e}\kappa _{e}z_{h}(e)}{\sum_{e}\kappa _{e}}.
\end{equation*}%
%
%
%
Choose a $C^{1}$ path $s\mapsto q_{s}\in \mathcal{V}_{N}$ with 
\begin{equation*}
q_{0}=q,\qquad \text{and}\qquad \left. \frac{d}{ds}q_{s}\right\vert _{s=0}=h,
\end{equation*}%
and continue the signed endpoints and common level by the implicit-function
construction of Section~\ref{sec:7}. We write $u_{\ell }^{\prime }[h]$, $%
v_{\ell }^{\prime }[h]$, and $\rho ^{\prime }[h]$ for the resulting first
derivatives at $s=0$. Their values depend only on $h$.

\begin{lemma}[Endpoint velocities]
\label{lem:endpoint-velocities} For $h\neq0$, the first derivatives of the
endpoints $u_l[h]$ and $v_l[h] $ are given by 
\begin{equation}  \label{eq:endpoint-velocity}
u_\ell^{\prime }[h] = \kappa_{u_\ell} \bigl(z_h(u_\ell)-\overline z_h\bigr),
\qquad v_\ell^{\prime }[h] = \kappa_{v_\ell} \bigl(\overline z_h-z_h(v_\ell)%
\bigr).
\end{equation}
\end{lemma}

\begin{proof}
Differentiate the two signed crossing equations in \eqref{eq:endpoint}.
Since 
\begin{equation*}
\sigma_{u_\ell}q^{\prime }(u_\ell)=-\lvert q^{\prime }(u_\ell)\rvert, \qquad
\sigma_{v_\ell}q^{\prime }(v_\ell)=\lvert q^{\prime }(v_\ell)\rvert,
\end{equation*}
we obtain 
\begin{equation*}
u_\ell^{\prime }[h] = \kappa_{u_\ell} \bigl(z_h(u_\ell)-\rho^{\prime }[h]%
\bigr),
\end{equation*}
\begin{equation*}
v_\ell^{\prime }[h] = \kappa_{v_\ell} \bigl(\rho^{\prime }[h]-z_h(v_\ell)%
\bigr).
\end{equation*}
Differentiating $\sum_\ell(v_\ell[h]-u_\ell[h])=\mu$ gives 
\begin{equation*}
0 = \sum_{\ell=1}^{m} \bigl(v_\ell^{\prime }[h]-u_\ell^{\prime }[h]\bigr) =
\rho^{\prime }[h]\sum_e\kappa_e - \sum_e\kappa_ez_h(e).
\end{equation*}
The weights are strictly positive, so 
\begin{equation*}
\rho^{\prime }[h] = \frac{\displaystyle\sum_e\kappa_ez_h(e)} {\displaystyle%
\sum_e\kappa_e} = \overline z_h.
\end{equation*}
Substitution gives \eqref{eq:endpoint-velocity}.
\end{proof}


Define the weighted boundary covariance 
\begin{equation}  \label{eq:Gamma}
\Gamma(h,k) = \sum_e\kappa_e \bigl(z_h(e)-\overline z_h\bigr) \bigl(%
z_k(e)-\overline z_k\bigr).
\end{equation}
If $K=\sum_e\kappa_e$, then 
\begin{equation*}
\Gamma(h,k) = \sum_e\kappa_ez_h(e)z_k(e) - \frac{ \left(\sum_e\kappa_ez_h(e)%
\right) \left(\sum_e\kappa_ez_k(e)\right)} {K}.
\end{equation*}
Thus the common-level constraint contributes one global rank-one correction
involving all endpoints. In particular, when $m\geq3$, all cross-component
terms are already contained in $\Gamma$; no blockwise independence is
assumed. The form is symmetric and positive semidefinite, since 
\begin{equation*}
\Gamma(h,h) = \sum_e\kappa_e \bigl(z_h(e)-\overline z_h\bigr)^2 \geq0.
\end{equation*}
We shall also use 
\begin{equation*}
\sum_e\kappa_e \bigl(z_h(e)-\overline z_h\bigr)=0.
\end{equation*}

\subsection{Cancellation of the boundary term 
}

Let $c\mapsto q_c$ be a $C^2$ family in $\mathcal{V}_N$ equipped with the
fixed-length common-level endpoint branch. Write 
\begin{equation*}
H_\ell=\left.\partial_\ell q_c\right|_{c=0}, \qquad
S_{\ell k}=\left.\partial_\ell\partial_kq_c\right|_{c=0},\qquad \ell,k=1,...,m.
\end{equation*}
Let 
\begin{equation*}
J(c)=\int_{G_c}q_c^2, \qquad D(c)=\int_\mathbb{T }q_c^2.
\end{equation*}

\begin{proposition}[The Hessian of the quotient]
\label{prop:general-hessian} At $c=0$, 
\begin{equation}  \label{eq:first-boundary-cancel}
\partial_\ell J(0) = 2\int_{F}q\, H_\ell.
\end{equation}
and 
\begin{equation}  \label{eq:general-hessian}
\partial_\ell\partial_k\mathcal{R}(0) = \frac{2}{D(0)} \left[ B(H_\ell,H_k) -
\rho\Gamma(H_\ell,H_k) + B(q,S_{\ell k}) \right].
\end{equation}
\end{proposition}

\begin{remark}[Why no endpoint accelerations occur]
\label{rem:no-acceleration} Because \eqref{eq:first-boundary-cancel} is an
exact identity along the whole branch, differentiating it once more requires
only 
endpoint velocities. Endpoint accelerations have not been discarded. If the
unsimplified transport formula is differentiated directly, they occur in the
combination $\rho(c)^2\partial_{\ell k}\lvert G_c\rvert$, which vanishes
identically.
\end{remark}

\begin{proof}
Here $q_0=q$, $G_0=F$, and $\rho(0)=\rho$. All endpoint quantities occurring
in $\Gamma$ are evaluated at $c=0$. Section~\ref{sec:7} shows that the
endpoint and common-level branches are real analytic. Hence $J$, $D$, and $%
\mathcal{R}=J/D $ are $C^2$, and $D(c)>0$ for all sufficiently small $c$.

Applying the explicit formula \eqref{eq:expl-integral} and the one-variable
Leibniz rule on every component gives 
\begin{align}
\partial_\ell J(c) &= 2\int_{G_c}q_c\,\partial_\ell q_c + \sum_{\ell=k}^{m} \left[
q_c\bigl(v_k(c)\bigr)^2\partial_\ell v_k(c) - q_c\bigl(u_k(c)\bigr)%
^2\partial_\ell u_k(c) \right]  \notag \\
&= 2\int_{G_c}q_c\,\partial_\ell q_c + \rho(c)^2 \sum_{k=1}^{m} \bigl(%
\partial_\ell v_k(c)-\partial_\ell u_k(c)\bigr)  \notag \\
&=2\int_{G_c}q_c\partial_\ell q_c+\rho(c)^2\partial_\ell |G_c|=2\int_{G_c}q_c%
\partial_\ell q_c.  \label{eq:first-boundary-cancel-2}
\end{align}
The last equality uses $\lvert G_c\rvert\equiv\mu$. Evaluating this equation
at $c=0$ then establishes \eqref{eq:first-boundary-cancel}.

{\ Differentiating \eqref{eq:first-boundary-cancel-2} again yields} 
\begin{align*}
\partial_\ell\partial_k
J(c)&=2\partial_k\int_{G_c}q_c\partial_\ell q_c=2\int_{G_c}\partial_\ell
q_c\partial_k q_c+2\int_{G_c}q_c\partial_\ell\partial_k q_c \\
&\qquad\qquad+2 \sum_{r=1}^{m} \left[ q_c\bigl(v_r(c)\bigr)%
\partial_\ell q_c\bigl(v_r(c)\bigr)\partial_k v_r(c) - q_c\bigl(u_r(c)%
\bigr)\partial_\ell q_c\bigl(v_r(c)\bigr)\partial_k u_r(c) \right].
\end{align*}
Evaluation at zero gives 
\begin{align*}
\partial_\ell \partial_k J(0)&=2\int_FH_\ell H_k+2\int_FqS_{\ell k}+
2\rho\sum_{r=1}^{m} \left[ \sigma_{v_r}H_i v_r^{\prime }[H_k] -
\sigma_{u_r}H_\ell u_r^{\prime }[H_k] \right] \\
&=2\int_FH_\ell H_k+2\int_FqS_{\ell k}+ 2\rho\sum_{r=1}^{m} \left[
z_{H_\ell}(v_r)v_r^{\prime }[H_k] - z_{H_\ell}(u_r)u_r^{\prime }[H_k] %
\right].
\end{align*}
Using Lemma~\ref{lem:endpoint-velocities} and $\sum_e\kappa_e \bigl(%
z_{H_k}(e)-\overline z_{H_k}\bigr)=0 $, the last summand becomes 
\begin{align*}
2\rho\sum_{r=1}^{m} \big[ z_{H_\ell}(v_r)v_r^{\prime }[H_k] &-
z_{H_\ell}(u_r)u_r^{\prime }[H_k] \big] = -2\rho\sum_e\kappa_ez_{H_\ell}(e) %
\bigl(z_{H_k}(e)-\overline z_{H_k}\bigr) \\
& =2 \sum_e\kappa_e \bigl(z_{H_\ell}(e)-\overline z_{H_\ell}\bigr) \bigl(%
z_{H_k}(e)-\overline z_{H_j}\bigr) + 2\overline z_{H_\ell} \sum_e\kappa_e \bigl(%
z_{H_k}(e)-\overline z_{H_k}\bigr) \\
& =2\Gamma(H_\ell,H_k).
\end{align*}
Summing up, we derived 
\begin{equation}  \label{eq:numerator-second}
\partial_\ell \partial_kJ(0) = 2\int_FH_\ell H_k + 2\int_FqS_{\ell k} -
2\rho\Gamma(H_\ell,H_k).
\end{equation}
Similarly, the denominator satisfies 
\begin{equation*}
\partial_\ell D(0)=2\int_\mathbb{T }qH_\ell\qquad\text{and}\qquad\partial_\ell%
\partial_kD(0) = 2\int_\mathbb{T }H_\ell H_k + 2\int_\mathbb{T }qS_{\ell k}.
\end{equation*}
At the base point, $\mathcal{R}(0)=J(0)/D(0)=\lambda$, and 
\begin{equation*}
\partial_\ell J(0)=2\int_FqH_\ell, \qquad \partial_\ell D(0)=2\int_\mathbb{T }qH_\ell.
\end{equation*}
Stationarity at $c=0$ gives $\partial_i\mathcal{R}(0)=0$ and therefore 
\begin{equation*}
\partial_\ell J(0)-\lambda\partial_\ell D(0) = 2B(q,H_\ell) = 0.
\end{equation*}
Using the product rule to differentiate $J=\mathcal{R }D$ finally shows 
\begin{equation*}
\partial_i\partial_k J(c) = D(c)\,\partial_\ell\partial_k\mathcal{R } (c)+
\partial_kD(c)\,\partial_\ell\mathcal{R}(c) + \partial_\ell D(c)\,\partial_k%
\mathcal{R}(c ) + \mathcal{R}(c)\,\partial_\ell\partial_k D.
\end{equation*}
At $c=0$, both first derivatives of $\mathcal{R}$ vanish and $\mathcal{R}%
(0)=\lambda$. Therefore 
\begin{equation*}
\partial_\ell\partial_k\mathcal{R}(0) = \frac{ \partial_\ell\partial_kJ(0) -
\lambda\partial_\ell\partial_kD(0)} {D(0)}.
\end{equation*}
Substitution gives \eqref{eq:general-hessian}.
\end{proof}

\subsection{Derivatives in the translation-invariant direction
}

For the block factor $Q_\ell$, write 
\begin{equation*}
A_\ell(t)=\frac{Q_\ell^{\prime }(t)}{Q_\ell(t)}
\end{equation*}
away from its roots. For the family \eqref{eq:block-translation}, we write
\begin{equation}  \label{eq:block-derivatives}
H_\ell=-qA_\ell, \qquad S_{\ell k}=qA_\ell A_k\quad(\ell\neq k), \qquad
H_\ell H_k=qS_{\ell k}.
\end{equation}
The apparent poles are removable in the actual products: 
\begin{equation*}
H_\ell = -CQ_\ell^{\prime }\prod_{j\neq\ell}Q_j,
\end{equation*}
\begin{equation*}
S_{\ell k} = CQ_\ell^{\prime }Q_k^{\prime } \prod_{j\neq\ell,k}Q_j.
\end{equation*}
Thus $H_\ell,S_{\ell k}\in\mathcal{V}_N$. Stationarity gives 
\begin{equation}  \label{eq:stationarity2}
B(q,S_{\ell k})=0, \qquad B(H_\ell,H_k)=B(q,S_{\ell k})=0
\quad(\ell\neq k),
\end{equation}
where for the second identity we used the definition of $B$ and the third
identity of \eqref{eq:block-derivatives}.
Moreover, set 
\begin{equation*}
\mathsf{H}=\left( \partial _{\ell }\partial _{k}\mathcal{R}(0)\right) _{\ell
,k},\qquad \boldsymbol{\Gamma }=\left( \Gamma (H_{\ell } ,H_{k})\right)
_{\ell ,k},
\end{equation*}%
and note that 
\begin{equation}
\mathsf{H}_{\ell k}=-\frac{2\rho }{D(0)}\boldsymbol{\Gamma }_{\ell k},\qquad
\ell \neq k,  \label{eq:H-offdiag}
\end{equation}
where we combined Proposition~\ref{prop:general-hessian} with %
\eqref{eq:stationarity2}.

Adding the same number to every block displacement translates the whole
configuration: 
\begin{equation*}
q_{c+s\mathbf{1}}(t)=q_c(t-s), \qquad \mathcal{R}(c+s\mathbf{1})=\mathcal{R}%
(c).
\end{equation*}
For sufficiently small $c$ and $s$, the translated endpoints 
\begin{equation*}
u_\ell(c,\rho(c))+s, \qquad v_\ell(c,\rho(c))+s
\end{equation*}
solve the signed crossing equations for $q_{c+s\mathbf{1}}$ at the common
level $\rho(c)$, and their total length remains $\mu$. Local uniqueness of
both implicit-function constructions gives 
\begin{equation*}
\rho(c+s\mathbf{1})=\rho(c),
\end{equation*}
\begin{equation*}
u_\ell(c+s\mathbf{1},\rho(c+s\mathbf{1})) = u_\ell(c,\rho(c))+s,
\end{equation*}
\begin{equation*}
v_\ell(c+s\mathbf{1},\rho(c+s\mathbf{1})) = v_\ell(c,\rho(c))+s.
\end{equation*}
Consequently $G_{c+s\mathbf{1}}=G_c+s$. Translation invariance of the circle
integrals proves 
\begin{equation*}
\mathcal{R}(c+s\mathbf{1})=\mathcal{R}(c).
\end{equation*}
Differentiating in $s$ gives 
\begin{equation*}
0=\partial_s\mathcal{R}(c+s\mathbf{1})
=\sum_{\ell=1}^m\partial_\ell\mathcal{R}(c+s\mathbf{1}).
\end{equation*}
Next evaluate at $s=0$, differentiate with respect to the block $k$ and set $%
c=0$. This shows 
\begin{equation*}
(\mathsf{H}\mathbf{1})_k=\sum_{\ell=1}^m \partial_\ell\partial_k\mathcal{R}%
(0)=0.
\end{equation*}
In other words 
\begin{equation}  \label{eq:H-neutral}
\mathsf{H}\mathbf{1}=0.
\end{equation}
This is an exact symmetry: only relative translations can contribute to the
second variation.

\section{Completion of the Proof of Theorem \protect\ref{thm:circle}}

\label{sec:9}

This section isolates the algebra that turns \eqref{eq:H-offdiag}--%
\eqref{eq:H-neutral} into a negative sum of squares for arbitrary $m$.

\subsection{Endpoint vectors}

At each outer endpoint, define 
\begin{equation*}
Y(e)=\left( \sigma _{e}H_{1}(e),\ldots ,\sigma _{e}H_{m}(e)\right)
=-\rho \left( A_{1}(e),\ldots ,A_{m}(e)\right) .
\end{equation*}%
Define $\overline{Y},x_\ell,y_\ell\in\mathbb{R}^m$ via 
\begin{equation*}
\overline{Y}=\frac{\sum_{e}\kappa _{e}Y(e)}{\sum_{e}\kappa _{e}},
\end{equation*}%
\begin{equation*}
x_{\ell }=Y(u_{\ell })-\overline{Y},\qquad \text{and}\qquad y_{\ell
}=Y(v_{\ell })-\overline{Y},
\end{equation*}%
and abbreviate 
\begin{equation*}
p_{\ell }=\kappa _{u_{\ell }},\qquad\text{and}\qquad s_{\ell }=\kappa
_{v_{\ell }}.
\end{equation*}
For every coordinate $\ell$, 
\begin{equation*}
\overline Y_\ell = \frac{\sum_e\kappa_e\sigma_eH_\ell(e)} {\sum_e\kappa_e}
= \overline z_{H_\ell}.
\end{equation*}
Thus the matrix assembled from the centered vectors $x_\ell,y_\ell$ is
precisely the boundary covariance matrix from Section~\ref{sec:8}, that is 
\begin{equation}  \label{eq:Gamma-matrix}
\boldsymbol{\Gamma }= \sum_{\ell=1}^{m} \left( p_\ell x_\ell x_\ell^{\mathsf{%
T}} + s_\ell y_\ell y_\ell^{\mathsf{T}} \right).
\end{equation}
Since 
\begin{equation*}
\sum_{\ell=1}^m H_\ell=-q\sum_{\ell=1}^m A_\ell=-q\sum_{\ell=1}^m \frac{%
Q_\ell^{\prime }}{Q_\ell}=-q^{\prime },
\end{equation*}
the endpoint orientations imply 
\begin{equation*}
\mathbf{1}\cdot Y(u_\ell) =-\sigma_{u_\ell}q^{\prime }(u_\ell) =\lvert
q^{\prime }(u_\ell)\rvert,
\end{equation*}
\begin{equation*}
\mathbf{1}\cdot Y(v_\ell) =-\sigma_{v_\ell}q^{\prime }(v_\ell) =-\lvert
q^{\prime }(v_\ell)\rvert.
\end{equation*}
Consequently, the contribution of component $\ell$ to $\mathbf{1}%
\cdot\sum_e\kappa_eY(e)$ is 
\begin{equation*}
p_\ell\,\mathbf{1}\cdot Y(u_\ell) +s_\ell\,\mathbf{1}\cdot Y(v_\ell) =1-1=0.
\end{equation*}
After summing over $\ell$, we obtain $\mathbf{1}\cdot\overline Y=0$.
Therefore, by definition 
\begin{equation}  \label{eq:xy-constraints}
p_\ell\,\mathbf{1}\cdot x_\ell=p_\ell\,\mathbf{1}\cdot (Y(u_\ell)-\overline{Y%
})=p_\ell|q^{\prime }(u_\ell)|=1.
\end{equation}
Similarly 
\begin{equation}  \label{eq:y}
s_\ell\,\mathbf{1}\cdot y_\ell=-1.
\end{equation}
Multiplying \eqref{eq:Gamma-matrix} by $\mathbf{1}$ and using %
\eqref{eq:xy-constraints}-\eqref{eq:y} gives 
\begin{align}
\boldsymbol{\Gamma}\mathbf{1 }&= \sum_{\ell=1}^{m} \left[ p_\ell x_\ell(%
\mathbf{1}\cdot x_\ell) + s_\ell y_\ell(\mathbf{1}\cdot y_\ell) \right] =
\sum_{\ell=1}^{m}(x_\ell-y_\ell).  \notag  \label{eq:Gamma-one}
\end{align}
Set 
\begin{equation*}
\widetilde{\mathsf{H}} = \frac{D(0)}{2\rho}\mathsf{H}.
\end{equation*}
For $\ell\neq k$, equation~\eqref{eq:H-offdiag} says 
\begin{equation*}
\widetilde{\mathsf{H}}_{\ell k}=-\boldsymbol{\Gamma}_{\ell k}.
\end{equation*}
Since $\widetilde{\mathsf{H}}\mathbf{1}=0$, every row sum is zero. Hence 
\begin{equation*}
\widetilde{\mathsf{H}}_{\ell\ell} = -\sum_{\ell\neq k}\widetilde{\mathsf{H}}%
_{\ell k} = \sum_{\ell\neq k}\boldsymbol{\Gamma}_{\ell k} = (\boldsymbol{%
\Gamma}\mathbf{1})_\ell-\boldsymbol{\Gamma}_{\ell\ell}.
\end{equation*}
Thus the off-diagonal formula and the neutral direction determine the whole
Hessian, entry by entry:

\begin{equation}  \label{eq:Htilde}
\widetilde{\mathsf{H}} = \text{diag}(\boldsymbol{\Gamma}\mathbf{1}) - 
\boldsymbol{\Gamma}.
\end{equation}

\subsection{An algebraic identity
}

Let $a=(a_1,\ldots,a_m)\in\mathbb{R}^m$ be a speed vector and put 
\begin{equation*}
w_\ell=a-a_\ell\mathbf{1}.
\end{equation*}
From \eqref{eq:Htilde}, 
\begin{align}
-a^{\mathsf{T}}\widetilde{\mathsf{H}}a &= a^{\mathsf{T}}\boldsymbol{\Gamma }%
a - \sum_{k=1}^m(\boldsymbol{\Gamma}\mathbf{1})_ka_k^2  \notag \\
&= \sum_{\ell=1}^{m} \left\{ p_\ell(x_\ell\cdot a)^2 + s_\ell(y_\ell\cdot
a)^2 + \sum_{k=1}^m\big(y_{\ell }-x_{\ell }\big)_ka_k^2 \right\}.
\label{eq:square-first}
\end{align}
Let us at first only consider the $\ell$-th summand. Since $a=w_\ell+a_\ell%
\mathbf{1}$, one has $a_k=(w_\ell)_k+a_\ell$ and the constraints %
\eqref{eq:xy-constraints} and \eqref{eq:y} give 
\begin{equation*}
x_\ell\cdot a =x_\ell\cdot w_\ell+\frac{a_\ell}{p_\ell}, \qquad y_\ell\cdot
a =y_\ell\cdot w_\ell-\frac{a_\ell}{s_\ell}.
\end{equation*}
Consequently, 
\begin{align*}
p_\ell(x_\ell\cdot a)^2 &+ s_\ell(y_\ell\cdot a)^2+ \sum_{k=1}^m\big(y_{\ell
}-x_{\ell }\big)_ka_k^2 =p_\ell (x_\ell\cdot w_\ell)^2+2a_\ell(x_\ell\cdot
w_\ell)+\frac{a_\ell^2}{p_\ell} \\
&+s_\ell (y_\ell\cdot w_\ell)^2-2a_\ell(y_\ell\cdot w_\ell)+\frac{a_\ell^2}{%
s_\ell} +\sum_{k=1}^m(y_\ell-x_\ell)_k\big((w_\ell)_k^2+2(w_\ell)_ka_\ell+a_%
\ell^2\big).
\end{align*}
The coefficient fo $a_\ell$ above is 
\begin{equation*}
2 \left[ x_\ell\cdot w_\ell -y_\ell\cdot w_\ell +(y_\ell-x_\ell)\cdot w_\ell %
\right]=0.
\end{equation*}
The coefficient of $a_\ell^2$ is 
\begin{equation*}
\frac1{p_\ell} +\frac1{s_\ell} +\mathbf{1}\cdot(y_\ell-x_\ell) =
\frac1{p_\ell}+\frac1{s_\ell} -\frac1{s_\ell}-\frac1{p_\ell} =0,
\end{equation*}
where we used \eqref{eq:xy-constraints} and \eqref{eq:y} again. Finally, $%
(w_\ell)_\ell=0$, and therefore 
\begin{equation*}
\sum_{k=1}^m(y_{\ell }-x_{\ell })_k(w_\ell)_k^2 = \sum_{k\neq\ell} (y_{\ell
}-x_{\ell })_k(a_k-a_\ell)^2.
\end{equation*}
Substitution in \eqref{eq:square-first} proves

\begin{equation}  \label{eq:square-completion}
-\frac{D(0)}{2\rho}a^{\mathsf{T}}\mathsf{H}a = \sum_{\ell=1}^{m} \left\{
p_\ell(x_\ell\cdot w_\ell)^2 + s_\ell(y_\ell\cdot w_\ell)^2 +
\sum_{k\neq\ell} (y_{\ell }-x_{\ell })_k(a_k-a_\ell)^2 \right\}.
\end{equation}

\subsection{The minimizing lower level set}

For the path $c=sa$, let $U_\ell$ and $V_\ell$ be the 
velocities of $u_\ell$ and $v_\ell$. Its tangent amplitude is 
\begin{equation*}
h_a=\left.\frac{d}{ds}q_{sa}\right|_{s=0}=\sum_{\ell=1}^{m}a_\ell H_\ell,
\end{equation*}
so 
\begin{equation*}
z_{h_a}(e)=a\cdot Y(e), \qquad\text{and}\qquad \overline
z_{h_a}=a\cdot\overline Y.
\end{equation*}
Lemma~\ref{lem:endpoint-velocities} gives 
\begin{equation*}
U_\ell=p_\ell\,a\cdot x_\ell, \qquad\text{and}\qquad V_\ell=-s_\ell\,a\cdot
y_\ell.
\end{equation*}
Because $a=w_\ell+a_\ell\mathbf{1}$, 
\begin{equation*}
U_\ell-a_\ell =p_\ell\,x_\ell\cdot w_\ell, \qquad\text{and}\qquad
V_\ell-a_\ell =-s_\ell\,y_\ell\cdot w_\ell.
\end{equation*}
Since $p_\ell^{-1}=\lvert q^{\prime }(u_\ell)\rvert$ and $s_\ell^{-1}=\lvert
q^{\prime }(v_\ell)\rvert$, squaring yields 
\begin{equation*}
p_\ell(x_\ell\cdot w_\ell)^2 = \lvert q^{\prime
}(u_\ell)\rvert(U_\ell-a_\ell)^2,
\end{equation*}
\begin{equation*}
s_\ell(y_\ell\cdot w_\ell)^2 = \lvert q^{\prime
}(v_\ell)\rvert(V_\ell-a_\ell)^2.
\end{equation*}
For $k\neq\ell$, we have 
\begin{align}
(y_{\ell }-x_{\ell })_k &= Y_k(v_\ell)-Y_k(u_\ell) = \rho \bigl(%
A_k(u_\ell)-A_k(v_\ell)\bigr).  \label{eq:cross-coefficient}
\end{align}
Use the fixed lift from Section~\ref{sec:5} for all roots and endpoints.
When $k\neq\ell$, the component $[u_\ell,v_\ell]$ contains no
representative, modulo $2\pi$, of a root from block $B_k$. Therefore $A_k$
is smooth on this closed interval. Differentiating its cotangent sum gives 
\begin{equation}  \label{eq:Aj-decreasing}
A_k^{\prime }(t) = -\frac14 \sum_{\theta\in B_k} \csc^2\frac{t-\theta}{2} <0,
\end{equation}
where roots are repeated according to multiplicity. Thus 
\begin{equation*}
A_k(u_\ell)-A_k(v_\ell)>0.
\end{equation*}
Combining these identities with \eqref{eq:square-completion} yields the
exact second-variation formula 
\begin{align*}
\left.\frac{d^2}{ds^2}\mathcal{R}(sa)\right|_{s=0} = -\frac{2\rho}{D(0)}
\sum_{\ell=1}^{m} \Bigg\{& \lvert q^{\prime }(u_\ell)\rvert(U_\ell-a_\ell)^2
+ \lvert q^{\prime }(v_\ell)\rvert(V_\ell-a_\ell)^2  \notag \\
&+ \rho\sum_{k\neq\ell} \bigl(A_k(u_\ell)-A_k(v_\ell)\bigr) (a_j-a_\ell)^2 %
\Bigg\}.  \label{eq:negative-squares}
\end{align*}
If $a$ is constant, then $q_{sa}(t)$ is merely a common translation of $q$,
and every term on the right vanishes, as it must. Suppose instead that $a$
is nonconstant. Choose indices $\ell\neq k$ with $a_\ell\neq a_k$. The
ordered-pair term 
\begin{equation*}
\rho \bigl(A_k(u_\ell)-A_k(v_\ell)\bigr) (a_k-a_\ell)^2
\end{equation*}
is then strictly positive; all remaining displayed terms are nonnegative.
Since $D(0)>0$ and $\rho>0$, it follows that 
\begin{equation}\label{neg:def}
\left.\frac{d^2}{ds^2}\mathcal{R}(sa)\right|_{s=0}<0.
\end{equation}
On the other hand, Section~\ref{sec:7} proves that $\mathcal{R}$ is $C^2$
near $0$, and \eqref{eq:R-local-min} makes $0$ a local minimum. Its Hessian
must therefore be positive semidefinite. This contradiction rules out $%
m\geq2 $, and we conclude that

\begin{equation*}
m=1.
\end{equation*}
Thus the minimizing lower-level set $F^*$ is a circle interval.

\subsection{Conclusion of the circle theorem}

\label{sec:10}

Let $I_{\mu }=F^{\ast }$ be the interval produced above and set 
\begin{equation*}
p_{q}(t)=e^{iNt/2}q(t)\in \mathcal{A}_{N}.
\end{equation*}%
Since $\lvert p_{q}\rvert ^{2}=q^{2}$, the admissible pair $(I_{\mu },p_{q})$
attains the global minimum, and therefore 
\begin{equation*}
\Lambda _{N}(\mu )=\frac{\int_{I_{\mu }}q^{2}}{\int_{\mathbb{T}}q^{2}}\geq
\lambda _{\min }\left( T_{I_{\mu }}^{(N)}\right) \text{.}
\end{equation*}%
Conversely, a lowest eigenfunction of $T_{I_{\mu }}^{(N)}$, together with
the same interval, is an admissible competitor in \eqref{eq:Lambda}. Hence 
\begin{equation*}
\Lambda _{N}(\mu )\leq \lambda _{\min }\left( T_{I_{\mu }}^{(N)}\right) .
\end{equation*}%
Therefore 
\begin{equation}
\Lambda _{N}(\mu )=\lambda _{\min }\left( T_{I_{\mu }}^{(N)}\right) .
\label{eq:Lambda-interval}
\end{equation}
For every $F\subset \mathbb{T}$ of measure $\mu $, 
\begin{equation*}
\lambda _{\min }\left( T_{F}^{(N)}\right) \geq \Lambda _{N}(\mu )\text{.}
\end{equation*}%
If $E=\mathbb{T}\setminus F$, then 
\begin{equation*}
\lambda _{\max }\left( T_{E}^{(N)}\right) =1-\lambda _{\min }\left(
T_{F}^{(N)}\right) \leq 1-\lambda _{\min }\left( T_{I_{\mu }}^{(N)}\right) 
\text{.}
\end{equation*}%
The complement of a circle interval is, up to rotation and null endpoints,
another circle interval. This proves Theorem~\ref{thm:circle}.

\section{Proof of Theorem~\protect\ref{thm:main}}

\label{sec:11}

\label{sec:deperiodization}

\textit{Step 1.}
First, we need to center the frequency interval. Modulation centers $\Omega $
without changing the concentration constant, and translation moves the
comparison interval without changing its concentration. We may therefore
assume 
\begin{equation*}
\Omega =[-B,B],\qquad B>0\text{.}
\end{equation*}%
The positive real-line concentration operator is 
\begin{equation*}
K_{B}(E)=M_{E}P_{[-B,B]}M_{E}\text{.}
\end{equation*}%
As $K_B(E)$ has the same eigenvalues as $K_{\lambda B}(E/\lambda)$ we assume
for simplicity that $B=1/2$ and write $K(E)=K_{1/2}(E)$. By %
\eqref{eq:operator-form}, 
\begin{equation*}
\mathfrak{C}_{[-1/2,1/2]}(E)=\lVert K(E)\rVert \text{.}
\end{equation*}%
On $L^{2}(E)$, this operator has kernel 
\begin{equation}
k(x-y),\qquad k(u)=\int_{-1/2}^{1/2}e^{2\pi i\xi u}\,d\xi =%
\begin{cases}
\dfrac{\sin (\pi u)}{\pi u}, & u\neq 0\text{,} \\[6pt] 
1, & u=0\text{.}%
\end{cases}
\label{eq:sinc-kernel}
\end{equation}%
In particular, $\lvert k(u)\rvert \leq 1$.

\textit{Step 2.} Now let us construct a dictionary for the expanding-circle limit. For $%
n\geq 1$, set 
\begin{equation}
L_{n}=2n+1\text{,}\qquad \text{and}\qquad \Delta _{n}=\frac{1}{M_{n}}.
\label{eq:expanding-parameters}
\end{equation}%
On the circle $\mathbb{R}/(L_{n}\mathbb{Z})$, let $\Pi _{n}$ be the
projection onto 
\begin{equation*}
\mathcal{H}_{n}=\text{span}_{\mathbb{C}}\{e^{2\pi ik\Delta _{n}x}:-n\leq
k\leq n\}\text{.}
\end{equation*}%
The unitary rescaling 
\begin{equation*}
(U_{n}f)(t)=\left( \frac{L_{n}}{2\pi }\right) ^{1/2}f\left( \frac{L_{n}t}{%
2\pi }\right)
\end{equation*}%
sends the symmetric modes $-n,\ldots ,n$ to $e^{-int},\ldots ,e^{int}$. If 
\begin{equation*}
\widetilde{E}_{n}=\frac{2\pi }{L_{n}}E\subset \mathbb{T}\text{,}
\end{equation*}%
then 
\begin{equation*}
U_{n}M_{E}U_{n}^{-1}=M_{\widetilde{E}_{n}},\qquad \lvert \widetilde{E}%
_{n}\rvert =\frac{2\pi }{L_{n}}\lvert E\rvert \text{.}
\end{equation*}%
Thus equality of the two sets measures is preserved under the rescaling.
Multiplication by $e^{int}$ then sends the symmetric modes to the analytic
modes $0,\ldots ,2n$ and commutes with all time indicators. The circle
theorem with degree $2n$ therefore gives 
\begin{equation}
\lVert \Pi _{n}M_{E}\Pi _{n}\rVert \leq \lVert \Pi _{n}M_{J}\Pi _{n}\rVert
\label{eq:large-circle-inequality}
\end{equation}%
whenever $E$ is measurable on $\mathbb{R}/(L_{n}\mathbb{Z})$ and $J$ is an
interval of the same measure. The projection kernel of $\Pi _{n}$ is 
\begin{equation}
k_{n}(u)=\Delta _{n}\sum_{k=-n}^{n}e^{2\pi i(k\Delta _{n})u}\text{.}
\label{eq:circle-kernel}
\end{equation}%
The points $k\Delta _{n}$ are exactly the midpoints of the $L_{n}$ equal
cells partitioning $[-1/2,1/2]$. Hence $k_{n}$ is the midpoint Riemann sum
for $k$. For $\lvert u\rvert \leq A$, a direct estimate gives 
\begin{align}
\sup_{\lvert u\rvert \leq A}\lvert k_{n}(u)-k(u)\rvert & =\sup_{|u|\leq
A}\left\vert \sum_{k=-n}^{n}\int_{\Delta _{n}(k-1/2)}^{\Delta _{n}(k+1/2)}%
\Big(e^{2\pi i\xi u}-e^{2\pi i(k\Delta _{n})u}\Big)d\xi \right\vert  \notag
\\
& \leq \sup_{|x|\leq A\Delta _{n}/2}\left\vert 1-e^{2\pi ix}\right\vert \leq 
\frac{\pi A}{2n+1}\longrightarrow 0\text{.}  \label{eq:kernel-uniform}
\end{align}%
This convergence is uniform only on bounded difference sets; the next step
is where it becomes operator-norm convergence.

\textit{Step 3.}
Now we prove the result for bounded time sets. Suppose that $E\subset (-R,R)$%
, and take $n$ so large that $L_{n}>2R$. Then $E$ embeds in the large circle
without wrap-around. Let 
\begin{equation*}
Q_{n}:\mathcal{H}_{n}\longrightarrow L^{2}(E)
\end{equation*}%
be the restriction operator. The operator $Q_{n}^{\ast }Q_{n}$ is the circle
concentration operator on $\mathcal{H}_{n}$, whereas $Q_{n}Q_{n}^{\ast }$ is
the integral operator $A_{n,E}$ on $L^{2}(E)$ with kernel $k_{n}(x-y)$. Thus 
\begin{equation}
\lVert \Pi _{n}M_{E}\Pi _{n}\rVert =\lVert A_{n,E}\rVert \text{.}
\label{eq:restriction-norm}
\end{equation}%
Let $A_{E}$ be the operator on $L^{2}(E)$ with kernel $k(x-y)$. It
represents $K(E)$, so 
\begin{equation*}
\lVert A_{E}\rVert =\lVert K(E)\rVert \text{.}
\end{equation*}%
Since $\lvert x-y\rvert \leq 2R$ for $(x,y)\in E\times E$, 
\begin{align}
\lVert A_{n,E}-A_{E}\rVert & \leq \lVert A_{n,E}-A_{E}\rVert _{\mathrm{HS}} 
\notag \\
& \leq \lvert E\rvert \sup_{\lvert u\rvert \leq 2R}\lvert
k_{n}(u)-k(u)\rvert \longrightarrow 0\text{.}  \label{eq:HS-convergence}
\end{align}%
Let $m_{E}=\lvert E\rvert $ and $I_{m_{E}}=(-m_{E}/2,m_{E}/2)$. If $E$ is
bounded, then $I_{m_{E}}$ is bounded as well. For all sufficiently large $n$%
, both sets embed in the same expanding circle. Apply %
\eqref{eq:large-circle-inequality}, then use \eqref{eq:restriction-norm}--%
\eqref{eq:HS-convergence} on both sides: 
\begin{equation}
\lVert K(E)\rVert \leq \lVert K(I_{m_{E}})\rVert \text{.}
\label{eq:bounded-real-line}
\end{equation}

\textit{Step 4.}
To conclude the proof of Theorem~\ref{thm:main} it only remains extending
the result to arbitrary finite measure sets.\ Let $E\subset \mathbb{R}$ be
measurable with $m_{E}=\lvert E\rvert <\infty $ and set 
\begin{equation*}
E_{R}=E\cap \lbrack -R,R],\qquad m_{R}=\lvert E_{R}\rvert \text{.}
\end{equation*}%
Then $m_{R}\uparrow m_{E}$. The full-space kernels of $K(E)$ and $K(E_{R})$
differ only on 
\begin{equation*}
(E\times E)\setminus (E_{R}\times E_{R})\text{,}
\end{equation*}%
whose measure is $m_{E}^{2}-m_{R}^{2}$. Since $\lvert k\rvert \leq 1$, 
\begin{equation}
\lVert K(E)-K(E_{R})\rVert \leq \sqrt{m_{E}^{2}-m_{R}^{2}}\longrightarrow 0%
\text{.}  \label{eq:truncation}
\end{equation}%
The same estimate applies to the nested centered intervals $I_{m_{R}}\subset
I_{m_{E}}$. Apply (\ref{eq:bounded-real-line}) to $E_{R}$: 
\begin{equation*}
\lVert K(E_{R})\rVert \leq \lVert K(I_{m_{R}})\rVert \text{.}
\end{equation*}%
Let $R\rightarrow \infty $ and use (\ref{eq:truncation}). We obtain 
\begin{equation*}
\lVert K(E)\rVert \leq \lVert K(I_{m_{E}})\rVert \text{,}
\end{equation*}%
which is (\ref{eq:main-concentration}). Taking square roots in (\ref%
{eq:operator-form}) yields the equivalent operator-norm formulation in
Theorem~\ref{thm:main} and the proof is concluded.
\section*{Acknowledgements}
We are extremely grateful
to Open AI and to the Large Language Model GPT 5.6 Sol.
This research was funded in part by the Austrian Science Fund (FWF) through
the projects 10.55776/PAT8205923 (L.D.A.) and 10.55776/PAT1384824 (M.S.).
For open access purposes, the authors have applied a CC BY public copyright
license to any author-accepted manuscript version arising from this
submission.

\end{document}